\documentclass[11pt]{amsart}
\usepackage{amsmath}
\usepackage{amsthm}

\usepackage{amsfonts, amssymb, stmaryrd}
\usepackage{amsbsy}
\usepackage[all]{xy}
\DeclareMathAlphabet{\mathpzc}{OT1}{pzc}{m}{it}

\usepackage[margin=2.2cm]{geometry}

\usepackage[retainorgcmds]{IEEEtrantools}

\usepackage{hyperref}
\hypersetup{colorlinks}
\usepackage{bm}

\usepackage{color} 

\definecolor{darkred}{rgb}{1,0,0} 
\definecolor{darkgreen}{rgb}{0,0.8,0}
\definecolor{darkblue}{rgb}{0,0,1}

\hypersetup{colorlinks,
linkcolor=darkblue,
filecolor=darkgreen,
urlcolor=darkred,
citecolor=darkgreen}

\newtheorem{lemma}[equation]{Lemma}
\newtheorem{theorem}[equation]{Theorem}

\newtheorem{proposition}[equation]{Proposition}
\newtheorem{corollary}[equation]{Corollary}
\newtheorem*{corollary*}{Corollary}

\theoremstyle{definition}
\newtheorem{definition}[equation]{Definition}

\newtheorem{remark}[equation]{Remark}
\newtheorem{example}[equation]{Example}

\numberwithin{equation}{section}

\newcommand{\fg}{\mathfrak g}
\newcommand{\fh}{\mathfrak h}
\newcommand{\fm}{\mathfrak m}
\newcommand{\fp}{\mathfrak p}
\newcommand{\fq}{\mathfrak q}

\newcommand{\cH}{{\mathcal H}}

\newcommand{\cO}{{\mathcal O}}

\newcommand{\R}{\mathbb{R}}

\newcommand{\XX}{\mathbb X}

\usepackage{mathrsfs} 

\newcommand{\scG}{\mathscr{G}}
\newcommand{\scH}{\mathscr{H}}
\newcommand{\scX}{\mathscr{X}_{mult}}

\newcommand{\stterm}{\mathsf{st2term}}
\newcommand{\stGpds}{\mathsf{stLinGpds}}
\newcommand{\twterm}{\mathsf{2term}}
\newcommand{\twVect}{\mathsf{2Vect}}
\newcommand{\inv}{^{-1}}
\newcommand{\toto}{\rightrightarrows}

\begin{document}

\title{Invariant vector fields and groupoids}

\author{Eugene Lerman}
  
\address{Department of Mathematics, University of Illinois, Urbana,
  IL 61801}

\begin{abstract} We use the notion of isomorphism between two
  invariant vector fields to shed new light on the issue of
  linearization of an invariant vector field near a relative
  equilibrium.  We argue that the notion is useful in understanding
  the passage from the space of invariant vector fields in a tube
  around a group orbit to the space invariant vector fields on a slice
  to the orbit.  The notion comes from Hepworth's study of vector
  fields on stacks.
\end{abstract}

\maketitle 

\tableofcontents

\section{Introduction}

Dynamics and bifurcation theory of 
group-invariant vector fields is an old and well-established area of
mathematics.  The literature on the subject is vast, and we will not
attempt to review it.  The area draws on a number of fields that
include representation theory, invariant theory, transformation
groups, singularity theory, equivariant transversality and geometric
theory of dynamical systems to name a few.

The goal of this paper is to add category theory to the arsenal of
tools.  More specifically we'd like to bring to the attention of the
dynamics community the notion of {\sf isomorphism} of invariant vector
fields and to show that it is useful and natural.  The source of the
idea lies in Hepworth's study of vector fields on stacks
\cite{Hep}. 

Traditionally one considers the collection $\Gamma(TM)^G$ of vector
fields on a manifold $M$ invariant under an action of a Lie group $G$
as a vector space with a topology, which is often a Whitney $C^\infty$
topology.  Note that multiplication by scalars in not continuous
\cite[pp.\ 46-47]{GG} in these topologies, so $\Gamma(TM)^G$ is a {\sf
  semi}-topological vector space.\footnote{The term ``semi-topological
  vector space'' is rarely used but it is convenient for our purposes.
  It is a also a somewhat ambiguous --- it is also used to describe
  vector spaces with a topology for which multiplication by scalars is
  continuous but addition is not.  We trust that our use of the term
  will not cause any undue confusion.}  On the other hand Hepworth
({\em op.\ cit.})  tells us that the space $\Gamma(TM)^G$ of invariant
vector fields is naturally the set of objects of a groupoid.  In this
paper we combine the two approaches by considering the space of
invariant vector fields as the space of objects in a linear
semi-topological groupoid which we denote by $\XX (G\times M\toto M)$
(see Definition~\ref{def:XX}).  This said, we strive to keep the paper
accessible to readers unfamiliar with groupoids and formulate most of
our results without mentioning groupoids explicitly.  We hope that a
reader can ignore all the references to groupoids and still profit
from reading the paper.

We give two applications of viewing invariant vector fields as objects of a
groupoid to equivariant dynamics.  The first one quantifies
non-uniqueness of linearization of invariant vector fields near
relative equilibria; we thus revisit the work of Krupa \cite{Krupa}.
The second deals with the notions of genericity for invariant vector
fields.
 
Given a manifold $M$ with a proper action of a Lie group $G$, consider
a $G$-invariant vector field $X\in \Gamma (TM)^G$.  Recall that a
point $x\in M$ is a {\sf relative equilibrium} of $X$ if the vector $X(x)$ is
tangent to the orbit $G{\cdot} x$.  Let $H$ denote the stabilizer of
$x$.  Following Krupa choose a slice $S$ through $x$ to the action of
$G$.  Then the restriction of
$X$ to the slice $S$ can be decomposed as
\begin{equation}
X|_S = X^\mathrm{h} + X^S  \label{eq.1}
\end{equation}
where $X^S$ is an $H$-invariant vector field on the slice and
$X^\mathrm{h}$ is tangent to the $G$-orbits (see also Lemma~8.5.3 in
\cite{Field}).  Note that $X^S$ vanishes at $x$. One can then deduce a
number of useful results about the dynamics of $X$ by studying the
dynamics of $X^S$ in a neighborhood of its equilibrium $x$.
Similarly, given a family of vector fields $X_\lambda$ one can analyze
the bifurcations of a relative equilibrium of $X_\lambda$ in terms of
the bifurcation of its projection $X_\lambda^S$ onto the slice.  For
all practical purposes one may think of the slice $S$ as a vector
space with a linear action of the compact group $H$.  We note that
dynamics and bifurcation theory of invariant vector fields on
representations have been studied intensely and extensively, and is
well understood. So reducing the study of invariant vector fields near
relative equilibria to the study of zeros of invariant vector fields
in representations is a natural thing to do.

However, neither the slice $S$ nor the decomposition \eqref{eq.1} for
a given choice of a slice are unique.  Nor is it clear that if a
vector field $X$ is generic then its projection $X^S$ is generic and
conversely.  Thus given two different choices of slices $S, S'$
through $x$ and $x' \in G\cdot x$ respectively and two choices of
splittings \eqref{eq.1}, it is far from clear how exactly the vector
fields $X^S$ and $X^{S'}$ are related (if they are related at all).
In particular there is no apparent relation between the spectra of the
linearizations $DX^S(x)$ and $DX^{S'}(x')$.  Consequently the notion
of the spectrum of a vector field at a relative equilibrium does not
seem to make sense (it is known the real part of the spectrum is
well-defined; see \cite[Lemma~8.5.2]{Field}).  This is one instance
where the notion of isomorphism of vector fields turns out to be
useful.

Given an action of a Lie group $G$ on a manifold $M$, as before,
consider the vector space 
\begin{equation}\label{eq:1}
C^\infty (M, \fg)^G =\{ \psi:M\to \fg \mid 
\psi(g\cdot m ) = Ad(g) \psi(m)\quad \textrm{ for all } m\in M, g\in G\}
\end{equation}
of Lie algebra valued equivariant maps.  This vector space maps
 into the space of invariant vector fields: we have a linear map
\begin{equation}\label{eq:chain-cx}
\partial: C^\infty (M, \fg)^G \to \Gamma(TM)^G, \quad \partial(\psi):= \psi_M
\end{equation}
where  the vector field $\psi_M$ is defined by
\begin{equation}\label{eq:1'}
\psi_M (m): = \left. \frac{d}{dt}\right|_{t=0} \exp (t\psi (m))\cdot m
\quad \textrm{ for all } m\in M.
\end{equation}
Following Hepworth (cf.\ \cite[Proposition~6.1]{Hep}) we have:
\begin{definition}[Isomorphic vector fields]\label{def:iso-vect}
  Two invariant vector fields $X, Y\in \Gamma(TM)^G$ are {\sf
    isomorphic} if there is an equivariant map $\psi\in C^\infty
  (M,\fg)^G$ with
\[
X = Y + \psi_M, 
\] 
where $\psi_M$ is defined by \eqref{eq:1'} above.
\end{definition}

The flows of two isomorphic vector fields are related by time-dependent ``gauge transformation.''  More precisely in Section~\ref{sec:2} below we prove:

\begin{theorem}\label{thm:1.0}
  Suppose two $G$-invariant vector fields $X$ and $Y$  on a manifold $M$ are
  isomorphic in the sense of Definition~\ref{def:iso-vect}.  Then
  there exists a family of maps $\{F_t:M\to G\}$ depending smoothly on
  $t$ so that flows $\Phi^X_t$, $\Phi^Y_t$ of $X$ and $Y$ respectively
  satisfy
\[
\Phi^X _t (m) = F_t (m)\cdot \Phi^Y_t (m) \quad 
\]
 for all $ (t,m) \in \R\times M$ for which $\Phi^X _t (m)$ is defined.
\end{theorem}

\noindent 
The purpose of Theorem~\ref{thm:1.0} is (1) to give readers a feel for
what it means dynamically for two vector fields to be isomorphic and
(2) to show that the notion of equivalence of vector fields has been
around implicitly in equivariant dynamics literature for quite some
time.  We then show that isomorphic vector fields occur naturally and
that the notion of isomorphism of vector fields is useful.
\begin{theorem}\label{thm:1.4}
  Let $X\in \Gamma(TM)^G$ be an invariant vector field, $x_1, x_2 $
  two points in $M$ on the same orbit, $S_1$, $S_2$ slices through the
  points $x_1, x_2$ respectively and $X^{S_i} \in \Gamma(TS_i)^{H_i}$
  the components of $X$ tangent to the corresponding slices ($H_i$ is
  the stabilizer of $x_i$).  Then, shrinking the slices if necessary,
  there exists an equivariant diffeomorphism $\varphi:S_1 \to S_2$ so
  that the vector fields $\varphi_* (X^{S_1})$ and $X^{S_2}$ are
  isomorphic.
\end{theorem}

\begin{remark}
  The decomposition \eqref{eq.1} implicitly used in the statement of
  Theorem~\ref{thm:1.4} is different from the one defined by Krupa:
  instead of using an invariant Riemannian metric on $M$ we use a left
  invariant connections on the principal bundles $G\to G\cdot x_i$,
  $i=1,2$.  
\end{remark}
\begin{theorem}\label{thm:1.5}
  Let $\rho:H\to GL(V)$ be a representation of a compact Lie group $H$
  and $X,Y:V\to V$ two $H$-invariant vector fields that are isomorphic,
  i.e., differ by a vector field $\psi_V$ induced by an equivariant
  function $\psi:V\to \fh$.  Suppose $X(0) =0$.  Then $Y(0) = 0$ and
\[
DX(0) = DY(0) + \delta \rho (\psi(0)).
\]
Here $\delta\rho:\fh\to
\mathfrak{gl}(V)$ is the corresponding representation of the Lie
algebra $\fh$ of $H$.
\end{theorem}

Together Theorems~\ref{thm:1.4} and \ref{thm:1.5} quantity the
ambiguity of linearization at relative equilibria:

\begin{theorem}\label{thm:1.100}
Let $X, X^{S_1}, X^{S_1}$ etc.\ be as in Theorem~\ref{thm:1.4}.  Then
\begin{equation}\label{eq:1.100}
D(X^{S_2}) (x_2) - T\varphi_{x_1}\, D (X^{S_1}) (x_1) \in 
\delta \rho(\psi(\fh_2^{H_2}))
\end{equation}
where $\rho:H_2 \to GL(T_{x_2}S_2)$ is the slice representation and
$\fh_2^{H_2}$ is the space of vectors in the Lie algebra of $H_2$
fixed by the adjoint representation.

In particular the linear map $T\varphi_{x_1}:T_{x_1}S_{x_1}\to
T_{x_2}S_{x_2}$ sends the eigenvectors of the linearization $D
(X^{S_1}) (x_1) $ to the eigenvectors of $D(X^{S_2}) (x_2)$, and the
real parts of the corresponding eigenvalues are the same.
\end{theorem}
\begin{remark}
  Theorem~\ref{thm:1.100} is sharper than
  \cite[Lemma~8.5.2]{Field}. For example, if the space $
  \rho(\psi(\fh_2^{H_2}))$ is zero then the $T\varphi_{x_1}$ defines
  an  isomorphism between the linearizations $D(X^{S_2}) (x_2)$ and $D
  (X^{S_1}) (x_1)$.  See Example~\ref{ex:1} below.
\end{remark}

\begin{example}\label{ex:1}
  Consider a particle of mass $m$ in 3-space subject to a central
  force field $F$.  Then the phase space of the system is $T\R^3
  \simeq \R^3 \times \R^3$, the equations of motion are of the form
\[
\left \{
\begin{array} {rcl}  \dot{q}&=& v\\
m\dot{v} &= &F(q,v),  
\end{array} 
\right.
\]
and the force $F:\R^3\times \R^3 \to \R^3 $ is $O(3)$-invariant:
\[
F(Aq, Av) = F(q,v)\quad \textrm{ for all }q,v\in \R^3,\,\, 
\textrm{ all }A\in O(3).
\]
The action of $O(3)$ on $\R^3\times \R^3$ has 3 orbit types. In more
detail, $(0,0)$ is the only fixed point.  For $(q,v)\in \R^3\times
\R^3$ with $q$ and $v$ linearly independent the stabilizer is
isomorphic to $\{\pm 1\} = ``O(1)"$.  If $q,v$ are two linearly
dependent vectors with $(q,v) \not = (0,0)$, then the stabilizer of
$(q,v)$ is isomorphic to $O(2)$.  In all cases the space $\fh^H$ of
fixed points in the Lie algebra $\fh$ of a stabilizer $H$ is zero.
Consequently  the linearization of the vector field
\[
X= \sum_{i=1}^3 \left( v_i\frac{\partial}{\partial q_i} + 
\frac{1}{m}F_i(q,v)
\frac{\partial}{\partial v_i}\right)
\]
at {\sf any} relative equilibrium is well-defined --- it doesn't matter
which slice through the relative equilibrium we pick, and it doesn't
matter which projection of the vector field $X$ onto the slice we
choose in order to compute the linearization.
\end{example}

\begin{remark}
  Theorem~\ref{thm:1.4} holds for families of vector fields: if
  $\{X_\lambda\}_{\lambda\in \Lambda}$ is a family of invariant vector
  fields then the vector fields $\varphi_* ({X_\lambda}^{S_1})$ and
  ${X_\lambda}^{S_2}$ are isomorphic with the isomorphism depending
  smoothly on $\lambda$.
\end{remark}

We next address the issue of genericity.  For the purposes of this
paper a set is {\em residual} if it is the intersection of countably
many open dense sets. A vector field on a manifold $M$ is {\em
  generic} if it belongs to a residual subset of $\Gamma(TM)$, Now
suppose once again that we have a proper action of a Lie group $G$ on
a manifold $M$, $H$ is the stabilizer of a point $x\in M$ and $S$ is a
slice through $x$.  We then have a canonical injective linear map
\begin{equation}\label{eq:1.8}
\Gamma (TS)^H\hookrightarrow \Gamma (T(G\cdot S))^G
\end{equation}
from the space of $H$-invariant vector field on the slice $S$ to the
space of $G$-invariant vector fields on the tube $G\cdot S$.  It is
clear that the image has infinite codimension.  So the notion of being
generic in these two spaces clearly don't correspond. A solution to
the problem is provided by:

\begin{theorem}\label{thm:1.11}
The map \eqref{eq:1.8} induces an invertible linear map 
\begin{equation}\label{eq:1.10}
  \Gamma (TS)^H/C^\infty (S, \fh)^H \to 
\Gamma (T(G{\cdot} S))^G/C^\infty (G{\cdot} S, \fg)^G
\end{equation}
from the space of isomorphism classes of invariant vector fields on
the slice to the space of isomorphism classes of invariant vector
fields on the tube.  Moreover if the spaces $ \Gamma (TS)^H, C^\infty
(S, \fh)^H, \Gamma (T(G{\cdot}S))^G$ and $C^\infty (G{\cdot}S, \fg)^G$
are given Whitney topologies, then the map \eqref{eq:1.10} is a
homeomorphism.  In other words it is an isomorphism of
semi-topological spaces.
\end{theorem}
We now make a reasonable assumption that if an invariant vector field
is generic than every vector field isomorphic to it is generic as
well. This is equivalent to requiring that the set of generic vector
fields $\Gamma (T(G{\cdot}S))^G$ lie in the preimage of a residual set
in the quotient $\Gamma (T(G{\cdot} S))^G/C^\infty (G{\cdot} S,
\fg)^G$.  With this assumption Theorem~\ref{thm:1.11} tells us that a
vector field $X\in \Gamma (TS)^H$ is generic if and only if its image
under the map \eqref{eq:1.8} is generic.  Thus Theorem~\ref{thm:1.11}
should be useful in equivariant bifurcation theory.  We plan to
address this elsewhere.

  Theorem~\ref{thm:1.11} is a decategorification of a stronger result,
  Theorem~\ref{thm:???}, which proves that the map \eqref{eq:1.8} is
  part of a chain homotopy equivalence of 2-term chain complexes of
  semi-topological vector spaces.

The paper is organized as follows.  In section~\ref{sec:2} we prove
Theorem~\ref{thm:1.0}.  In section~\ref{sec:3} we prove
Theorem~\ref{thm:1.4}, Theorem~\ref{thm:1.5} and Theorem~\ref{thm:1.100}.  In
section~\ref{sec:4} we prove Theorem~\ref{thm:???}.  In the last
section, section~\ref{sec:groupoid}, we explain where the notion of
isomorphic vector fields comes from and describe the connections with
groupoids and stacks.

\section{Isomorphic vector fields and their
  flows }\label{sec:2}

The goal of this section is to prove Theorem~\ref{thm:1.0}.
Throughout the section a Lie group $G$ with Lie algebra $\fg$ acts
properly on a manifold $M$.  We do not assume that $G$ is compact. For
a point $m\in M$ we have the evaluation map
\[
ev_m: G\to M,\quad ev_m (a) = a\cdot m\quad \textrm{ for all }a\in G.
\]
Dually, for any $a\in G$ we have a diffeomorphism
\[
a_M :M\to M, \quad a_M (m) = a\cdot m \quad \textrm{ for all } m \in M.
\]
Recall that every $\xi \in \fg$ we have the induced vector field $\xi_M\in \Gamma (TM)$ defined by 
\begin{equation}\label{eq:xi_M}
\xi_M (m) = \left. \frac{d}{dt}\right|_0 \exp(t\xi)\cdot m.
\end{equation}
Note that by the chain rule
\[
 \left. \frac{d}{dt}\right|_0 \exp(t\xi)\cdot m =  \left.\frac{d}{dt}\right|_0 (ev_m (\exp t\xi)) =
  T(ev_m)_0\,\, \xi.
\]
Thus 
\begin{equation}
\xi_M (m) = T(ev_m)_0 \xi.
\end{equation}
\begin{remark}
We trust that the reader will have no difficulty distinguishing between the induced vector fields $\xi_M$ \eqref{eq:xi_M} and $\psi_M$ \eqref{eq:1'}.
\end{remark}

\begin{proof}[Proof of Theorem~\ref{thm:1.0}]
  Recall that we have a manifold $M$ with an action of a Lie group
  $G$, $X,Y\in \Gamma(TM)^G$ are two $G$-invariant vector fields and
  $\psi: M\to \fg$ is a $G$-equivariant function with
\[
X = Y + \psi_M.
\]
Fix a point $z\in M$. Then 
\[
\gamma (t) := \Phi_t^Y (z)
\]
 is an integral
curve of $Y$ through $z$.  We'd like to prove the existence of a curve
$g(t)$ in $G$ which depends smoothly on $z$ so that
\[
\sigma (t):= g(t) \cdot \gamma (t)
\]
is an integral curve of the vector field $X$ through $z$.  It is
well-known (see, for example, \cite[Proposition~1.13.4]{DK}) that for
any smooth curve $\tau:I\to \fg$ in the Lie algebra $\fg$ there is a
unique smooth curve $g:I\to G$ defined on the {\sf same} interval $I$
so that $g(0) = e$ and $g(t)$ solves the ODE
\begin{equation}
\dot{g}(t) = TR_{g(t)}\, \tau (t).
\end{equation}
Here and elsewhere in the paper $R_a:G\to G$ denotes the right
multiplication by $a\in G$.  
In our case we set
\[
\tau(t) = \psi (\gamma (t)) = \psi (\Phi^Y_t (z))
\]
We now check that the solutions $g(t)$ of the ODE
\begin{equation}
\dot{g}(t) = TR_{g(t)} \psi (\Phi^Y_t (z)), \quad g(0) =e,
\end{equation}
which depend smoothly on $z$ define the desired time-dependent family
of maps $\{F_t:M\to G\}$.  For that it is enough to check that 
\[
\sigma(t) := g(t) \cdot \gamma (t)
\]
is the integral curve of $X$ through $z$.  By the chain rule
\[
\left. \frac{d}{dt}\right|_0 (g(t)\cdot \gamma (t)) = T \left(g(t)_M
\right) _{\gamma (t)} \dot{\gamma} (t) + T\left( ev_{\gamma
    (t)}\right) _{g(t)} \dot{g}(t).
\]
By definition of $\gamma(t)$,
\[
T \left(g(t)_M
\right) _{\gamma (t)} \dot{\gamma} (t) 
= T \left(g(t)_M \right) _{\gamma (t)} Y(\gamma(t)) = Y(g(t)\cdot \gamma (t)) =
Y(\sigma(t)),
\]
where the next to last equality holds by $G$-invariance of the vector field
$Y$. On the other hand,
\begin{align*}
  T\left( ev_{\gamma (t)}\right) _{g(t)} \dot{g}(t) &= T\left(
    ev_{\gamma
      (t)}\right) _{g(t)} T R_{g(t)}\psi(\gamma (t)) 
\quad\textrm{ by definition of } g(t)\\
  &=\left. \frac{d}{ds}\right|_0 \left(g(t) \exp
    (s\psi(\gamma(t)))\right)
  \cdot \gamma (t)\\
  &=\left. \frac{d}{ds}\right|_0 \left(g(t) \exp (s\psi(\gamma(t)))g(t)\inv\right) \cdot g(t) \cdot\gamma (t)\\
    &=\left. \frac{d}{ds}\right|_0 ( \exp (sAd(g(t))\psi(\gamma(t)))\cdot g(t) \cdot \gamma (t)\\
    &=\left. \frac{d}{ds}\right|_0 \left( \exp (s\psi(g(t)\cdot\gamma(t)))\,\right)\cdot (g(t) \cdot \gamma (t))\quad \textrm{ by equivariance of } \psi\\
    &= \psi _M (g(t)\cdot \gamma (t)) =  \psi _M (\sigma(t)).
\end{align*}
Therefore
\[
\left. \frac{d}{dt}\right|_0 \sigma (t) = Y(\sigma(t)) + \psi_M (\sigma(t)) = X (\sigma(t))
\]
and we are done.
\end{proof}
\begin{remark}
  Something about the proof of the theorem may look vaguely familiar
  to some readers.  Indeed it looks very much like the reconstruction
  argument in symplectic reduction theory.  See for instance the
  discussion at the bottom of p.~304 in \cite{AM}.  From the point of
  view of Theorem~\ref{thm:1.0} the reconstruction argument says that
  the horizontal lift of the reduced vector field with respect to some
  connection and the original Hamiltonian vector field are isomorphic.
\end{remark}

\begin{definition}
  Since the orbit $G\cdot x$ through a relative equilibrium $x$ of a
  vector field $X$ is preserved by the flow of $X$ it makes sense to
  define the relative equilibrium $x$ to be {\sf hyperbolic} if the
  manifold $G\cdot x$ is normally hyperbolic for the flow of $X$.
\end{definition}

\begin{corollary}
  Suppose $X,Y\in \Gamma (TM)^G$ are two isomorphic invariant vector
  fields. Then their flows induce the {\sf same} flow on the space of
  orbits $M/G$.  

Consequently isomorphic vector fields have the same
  relative equilibria and relative periodic orbits.  Moreover if $x$
  is a hyperbolic relative equilibrium for $X$ it is hyperbolic for
  any vector field $Y$ isomorphic to $X$.
\end{corollary}

\begin{remark}
  The first part of the corollary has a converse if the action of $G$
  on $M$ is free and proper: if $X$ and $Y$ are two $G$-invariant
  vector fields inducing the same flow on the orbit space $B:= M/G$
  then $X$ and $Y$ are isomorphic.

  The argument proceeds as follows. Since the action of $G$ is free and proper,  the orbit space $B:= M/G$ is a manifold and the
  orbit map $\pi:M\to B$ makes $M$ into a principal $G$-bundle over
  $M$.  Then two $G$-invariant vector fields $X$ and $Y$ induce the same
  flow on $B$ if and only if they are $\pi$-related to the same vector
  field on $B$ if and only if $T\pi (X-Y) =0$ if and only if for every
  $m\in M$ there is a vector $\psi(m)\in \fg$ so that $X(m) -Y(m) =
  (\psi(m))_M (m)$.  It is also easy to see that the function $\psi:M\to \fg$
   is smooth:  
\[
\psi(m) = \alpha _m\, (X(m) - Y(m)),
\]
where $\alpha \in
  \Omega^1(M, \fg)^G$ is a connection 1-form  (any choice of $\alpha$ will do). 
Thus $X$ and $Y$ are isomorphic with the isomorphism provided by the
$G$-equivariant map $\psi$ defined above.
\end{remark}

\section{Local normal form for an invariant vector field} \label{sec:3}

Once again let $S$ denote a slice through a point $x\in M$ for a
proper action of a Lie group $G$ on a manifold $M$ and $H$ denote the
stabilizer of $x$.  Then
\[
G\cdot S := \{g\cdot y\mid g\in G, y\in S\}
\]
is an open $G$-invariant neighborhood of the orbit $G\cdot x$, which
is often called a tube.  We have a $G$-equivariant diffeomorphism $\varphi$ from
the associated bundle
\[
G\times ^H S := (G\times S)/H
\]
to the tube $G\cdot S$ making the diagram
\[\xy
(-10,10)*+{G\times^H S}="1";
(20,10)*+{G\cdot S}="2";
(-10,-7)*+{G/H}="3";
(20,-7)*+{G\cdot x}="4";
 {\ar@{->}^{\varphi} "1";"2"};
 {\ar@{->}^{\pi} "2";"4"};
 {\ar@{->}_{} "1";"3"};{\ar@{->}_{\pi} "3";"4"};
\endxy\]
commute.  Here 
the left vertical map is 
\[
[g,s]\mapsto gH,
\]
the bottom horizontal map is 
\[
gH\mapsto g\cdot x
\]
and the right vertical map $\pi:G{\cdot} S\to G\cdot x$ is given by
\[
\pi (g\cdot y) = g\cdot x;
\] 
it is well-defined.
\begin{remark}
  The projection $\pi:G\cdot S \to G\cdot x$ very much depends on the
  choice of the slice $S$: the fiber of $\pi$ above $g\cdot x\in
  G\cdot x$ is the submanifold $g\cdot S$, which is a slice through
  $g\cdot x$.  A choice of a different slice $S'$ through $x$ with
  $G{\cdot }S' = G\cdot S$ defines a {\sf different} submersion
  $\pi':G{\cdot}S' \to G\cdot x$, even though it is given by a
  seemingly identical formula:
\[
\pi'(g\cdot y') = g\cdot x
\]
for all $g\in G$, $y' \in S'$.  In particular $\pi$ and $\pi'$ have
different fibers.
\end{remark}

\begin{lemma}\label{lem:3.2}
  Let $M$ be a manifold with an action of a Lie group $G$, $H$ the
  stabilizer of a point $x\in M$ and $S$ a slice through $x$ for the
  action of $G$.  A choice of an $H$-equivariant splitting 
\begin{equation}\label{eq:split}
\fg = \fh \oplus \fm
\end{equation}
of the Lie algebra $\fg$ of $G$ into the Lie algebra $\fh$ of $H$ and
a complement $\fm$ gives rise to an isomorphism of vector spaces
\[
\Gamma (T(G{\cdot}S))^G \to \Gamma (TS)^H \oplus C^\infty (S,\fm)^H.
\]
\end{lemma}

\begin{proof}
  It will be convenient for notational purposes to assume that
  $G{\cdot S} =M$.  Then, since $S$ is a global slice, any
  $G$-invariant vector field $X$ on $M$ is uniquely determined by its
  restriction to $S$.  Thus the restriction map
\begin{equation}
\Gamma (TM)^G \to \Gamma (TM)^G |_S, \quad X\to X|_S
\end{equation}
is an isomorphism of vector spaces.  

The splitting \eqref{eq:split} defines a $G$-invariant connection on
the principal $H$-bundle $G\to G\cdot x$ and consequently a
$G$-invariant connection on the associated bundle $M \simeq G\times ^H
S \to G\cdot x$.  Hence any $G$-invariant vector field $X$ on
$M$ can be uniquely written as a sum
\begin{equation}\label{eq:3.5}
X = X^\mathrm{v}+ X^\mathrm{h}
\end{equation}
of two $G$-invariant vector fields with $X^\mathrm{v}$ being tangent
to the fibers of $\pi:M\to G\cdot x$ and $ X^\mathrm{h}$ being tangent
to the horizontal distribution $\mathcal{H}\subset TM$ defined by the
splitting \eqref{eq:split}. It follows that the restriction
\[
X^S: = X^\mathrm{v}|_S
\]
is tangent to $S$. It is easy to see that $X^S$ is $H$-invariant.

The restriction $\cH|_S$ of the horizontal distribution to the slice
is trivial: an isomorphism is given by 
\begin{equation}
S\times \fm \to \cH|_S, \quad (y,\xi)\mapsto \xi_M (y).
\end{equation}
It follows that the space of $H$-equivariant sections of $\cH|_S \to
S$ is isomorphic to the space of $H$-equivariant function $C^\infty
(S, \fm)^H$:
\begin{equation}
  \Gamma(\cH|_S )^H \simeq  C^\infty (S, \fm)^H,
\end{equation}
and the result follows.
\end{proof}

Lemma~\ref{lem:3.2} tells us that a choice of a slice $S$ and of a
splitting \eqref{eq:split} defines a surjective linear map 
\begin{equation}\label{eq:3.8}
  \fp: \Gamma (TG{\cdot}S)^G \to C^\infty (S, \fm)^H, 
\quad
  X\mapsto \psi^S_X.
\end{equation}

\begin{remark}\label{rmrk:3.11}
  The function $\psi_X^S \in C^\infty (S, \fm)^H$ and the vector field
  $X^\mathrm{h}$ in \eqref{eq:3.5} are, of course, directly related.
  Indeed, define $\Psi^X \in C^\infty (G{\cdot}S), \fg)^G$ by
\[
\Psi_X (g\cdot y) =Ad(g) \psi_X^S (y).
\]
Then
\[
X^\mathrm{h} = (\Psi_X)_M (m)
\]
for all $m\in G{\cdot} S$, and the decomposition  \eqref{eq:3.5} reads:
\begin{equation}\label{eq:3.12}
X (g\cdot y) = T (g_M)_y \, X^S (y) +  (\Psi_X)_M (g\cdot y)
\end{equation}
for all $y\in S$, $g\in G$.
\end{remark}

\begin{lemma}\label{lem:2.3}
  Let $S, S'$ be two different slices through the same point $x$ for a proper
  action of a Lie group $G$ on a manifold $M$.  Let $H$ denote the
  stabilizer of $x$, as before.  There exists an $H$-equivariant
  $G$-valued function $f$ defined on a neighborhood $U$ of $x$ in $S$
  so that
\[
\varphi:U\to S', \quad \varphi(y) = f(y)\cdot y
\]
is an $H$-equivariant open embedding.  Moreover we may assume  $f$ takes values in $\exp{\fm}$, where $\fg= \fh\oplus \fm$ is an $H$-equivariant splitting.
\end{lemma}
\begin{proof}
  We may assume that $M= G \times ^H S$ and $S'$ is a slice
  through $[1,x]\in G\times ^H S$, where $[1,x]$ is the $H$-orbit of
  $(1,x)\in G\times S$.  It is a standard fact that an $H$-invariant
  splitting $\fg = \fm \oplus \fh$ defines an $H$-equivariant section
  $s$ of the principal $H$ bundle $G\to G/H$.  Explicitly the section
  is given by the formula
\[
s(gH) = g
\]
for all $g\in \exp (\fm)$ sufficiently close to $1$, say for $g$ in an
$H$-invariant neighborhood $\cO$ of $1\in \exp(\fm)$.  

The section $s$ trivializes the associated bundle $\pi: G\times^H S\to G/H= G{\cdot} x$.
Explicitly the trivialization is the map 
\begin{equation}
\cO\times S\to \pi\inv (\cO G), \quad (g, y)\mapsto [g,y] = g{\cdot}[1,y].
\end{equation}
We may assume that $S'\subset \cO\times S$ and that at $(1,x) \in \cO\times S$ we have
\[
T_{(1,x)}(\cO \times S) =   T_{(1,x)}\cO\oplus T_{(1,x)}S'. 
\]
Consequently the differential of $pr_2| _{S'}:S'\to S$ at $(1,x)$ is
an isomorphism, hence a diffeomorphism from a neighborhood $U'$ of
$(1,x)\in S'$ to a neighborhood $U$ of $x$ in $S$.  Since $pr_2$ is
$H$-equivariant, $pr_2| _{S'}$ is $H$-equivariant as well, and we may
take $U$, $U'$ to be $H$-invariant.  We set
\[
\varphi = (pr_2|_{U'})\inv :U\to U'.
\]
It is the desired map, and it's of the form
\[
\varphi(y) = (f(y), y)
\]
for some $H$-equivariant map $f:U\to \cO \subset \exp (\fm)$.

When we identify $G\times ^H S$ with $G{\cdot}S$ the map $\varphi$
takes the form
\[
\varphi (y) = f(y)\cdot y,
\] 
as desired.
\end{proof}

\begin{lemma}\label{lem:3.15}
  Let $\varphi:S\to S'$, $\varphi(y) = f(y){\cdot}y$ be the
  equivariant map of Lemma~\ref{lem:2.3}.  Denote the left
  multiplication by an element $a\in G$ by $L_a$. For any $y\in S$,
  $v\in T_yS$,
\begin{equation}\label{eq:3.13}
  T\varphi_y (v) = T (f(y)_M) _y 
\left[(TL_{f(y)\inv} (Tf_y (v)))_M (\varphi(y)) + v\right].
\end{equation}
\end{lemma}

\begin{remark} In \eqref{eq:3.13} $f(y)\in G$, $f(y)_M:M\to M$ is the
corresponding diffeomorphism, and $T (f(y)_M) _y: T_y M\to
T_{f(y)}M$ is its differential.  Similarly $Tf_y: T_yM\to T_{f(y)}G$ is
the differential of $f$ at $y$, $TL_{f(y)\inv} (Tf_y (v)) \in T_1 G =
\fg$, and $(TL_{f(y)\inv} (Tf_y (v)))_M (\varphi(y))$ is the value at
$\varphi(y)$ of the vector field on $M$ induced by the vector
$TL_{f(y)\inv} (Tf_y (v))\in \fg$.
\end{remark}

\begin{proof}[Proof of Lemma~\ref{lem:3.15}] The derivative $T
  (ev_y)_g$ of the evaluation map 
\[
ev_y:G\to M, \quad ev_y (g) = g\cdot y
\]
at a point $g\in G$ can be computed as follows:  For $w\in T_g G$ set 
\[
z= TL_{g\inv} w\in T_eG = \fg.
\]
Then $w = TL_g z$ and 
\begin{equation}\label{eq:3.19}
  T(ev_y) w = \left.\frac{d}{dt}\right|_0 ev _y (L_g (\exp tz)) =
  \left.\frac{d}{dt}\right|_0 g(\exp tz) y= T(g_M)_y (T L_{g\inv}
  w)_M (y).
\end{equation}
Next choose a curve 
$\gamma:I\to M$ with $\gamma (0) = y$, $\dot{\gamma} (0)  = v$.
Then
\begin{align*}
  T\varphi_y (v)& = \left.\frac{d}{dt}\right|_0 \varphi (\gamma (t)) =
  \left.\frac{d}{dt}\right|_0 (f(\gamma(t)){\cdot}\gamma(t))=\\
  &= T(ev _{\gamma (0)} ) \left.\frac{d}{dt}\right|_0 f(\gamma(t)) 
+ T(f(\gamma (0)))_M (\dot{\gamma}(0))\\
  &= T(f(y)_M)_y \left( (TL_{g\inv}Tf_y (v))_M (y) + v\right).
\quad \textrm{ by \eqref{eq:3.19}}
%
\end{align*}
\end{proof}
\begin{proof}[Proof of Theorem~\ref{thm:1.4}]
We first address the dependence of the projection
\[
\fq: \Gamma (T(G{\cdot}S))^G\to \Gamma (TS)^H
\] 
on the choice of the splitting $\fg=\fh\oplus \fm$.

\begin{lemma}
Let 
\begin{equation} \label{eq:3.17}
\fg = \fh \oplus \fm_1 = \fh \oplus \fm_2
\end{equation}
be two $H$-equivariant splittings and 
\begin{equation}
\fq_1,\fq_2: \Gamma (T(G{\cdot}S))^G\to \Gamma (TS)^H
\end{equation}
the two corresponding projections. Then for any $X\in \Gamma
(T(G{\cdot}S))^G$ the two vector fields
\[
X_i := \fq_i (X), \quad i=1,2
\]
are isomorphic in the sense of Definition~\ref{def:iso-vect}: there is
map $\varphi \in C^\infty (S, \fh)^H$ with
\[
X_1 -X_2 = \varphi_S.
\]
\end{lemma}

\begin{proof}
Since $\fm_1$, $\fm_2$ are both complementary to $\fh$ in $\fg$ there exists a linear map 
\[
B: \fm_1 \to \fh
\]
so that $\fm_2$ is the graph of $B$:
\[
\fm_2 =\{ v +B(v) \in \fg \mid v\in \fm_1\}.
\]
Since the splittings \eqref{eq:3.17} are $H$-equivariant, the map $B$
is $H$-invariant.  Define $\psi_i \in C^\infty (S, \fm_i)^H$,
$(i=1,2)$, by
\[
\psi_i (y) := \left(T(ev_x)_0|_{\fm_i}\right)\inv \left( T\pi_y (X(y)\right).
\]
Then
\[
\psi_1 (y) + B(\psi_1 (y)) \textrm{ is in } \fm_2 \textrm{ for all } y\in S.
\]
Since $B(\psi_1(y))\in \fh$, we have 
\[
T(ev_x)_0 (B(\psi_1 (y))) = 0\quad \textrm{ for all }y\in S.
\]
Hence
\[
T(ev_x)_0 (\psi_1 (y) + B(\psi_1 (y)))=T(ev_x)_0 (\psi_1 (y))  = T\pi_y (X(y))
= T(ev_x)_0  (\psi_2 (y)).
\]
Since $T((ev_x)_0|_{\fm_2}$ is 1-1, it follows that 
\[
\psi_1 (y) + B(\psi_1 (y))=\psi_2 (y).
\]
We {\sf define} $\varphi:S\to \fh$ by
\[
\varphi(y) = B(\psi_1 (y));
\]
it is $H$-equivariant. Then
\[
\psi_2 - \psi_1 = \varphi.
\]
Moreover
\[
(X_1 -X_2)(y) = (X(y) - (\psi_1)_M (y)) - (X(y) - (\psi_2)_M (y))=
(\psi_2 -\psi_1)_M (y) = \varphi_M (y), 
\]
which proves the lemma.
\end{proof}
Next we address the dependence of the projection
\[
\fq: \Gamma (T(G{\cdot}S))^G \to \Gamma (TS)^H
\]
on the point $x$.  For any $g\in G$ 
\[
S':= g_M (S) \quad (= g\cdot S)
\]
is a slice through $g\cdot x$, the stabilizer of $g\cdot x $ is $H' =
gHg\inv$, and an $H$-equivariant splitting $\fg = \fh \oplus \fm$
defines an $H'$-equivariant splitting
\begin{equation}
\fg = Ad(g)\fh\oplus Ad(g) \fm = \fh'\oplus Ad(g)\fm.
\end{equation}
Since $g_M:S\to S'$ is an $H-H'$-equivariant diffeomorphism, it
defines an isomorphism
\[
g_*: \Gamma (TS)^H \to \Gamma (TS')^{H'}, 
\]
given by 
\[
(g_*Y)(y') = Tg_M Y(g\inv \cdot y'), \quad y'\in S'.
\]
Equation \eqref{eq:3.12} now implies that the image $X^{S'}$ of $X\in
\Gamma (T(G{\cdot S}))^G$ under the projection
\[
\fq': \Gamma (T(G{\cdot S}))^G \to \Gamma (TS')^{H'}
\]
(which is defined by the slice $S'$ and the splitting $\fg= \fh'
\oplus Ad(g)\fm$) is exactly $g_* X^S$.  In other words, in this case
\[
X^{S'} \textrm{ and } g_* X^S \textrm{ are equal.}
\]
Therefore, to finish our proof of Theorem~\ref{thm:1.4} it is enough
to show:
\begin{lemma}\label{lem:3.22}
  Suppose $S, S'$ are two slices through the {\sf same} point $x$. Let
  $\varphi:S\to S'$, $\varphi(y) = f(y)\cdot y$ be the equivariant
  diffeomorphism of Lemma~\ref{lem:2.3}.  Then for any $G$-invariant
  vector field $X$ on the original manifold $M$ the vector fields $X^{S'}$
  and $\varphi_* X^S \in \Gamma (TS')^H$ are isomorphic: there exists
  an $H$-equivariant map $\nu: S'\to \fh$ with
\[
X^{S'} - \varphi_*X^S = \nu_S.
\]
Here $\nu_S (y') = \frac{d}{dt}|_0 (\exp t\nu(y')) \cdot y'$ and $
\varphi_*X^S (y') = T\varphi_{\varphi\inv (y')}X^S (\varphi\inv (y'))$
for $y'\in S'$.
\end{lemma}

\begin{proof} As before fix an $H$-equivariant splitting $\fg=
  \fh\oplus \fm$. Then by Remark~\ref{rmrk:3.11} for any point $y$ in
  the slice $S$ we have 
\[
X^S (y) = X (y) - (\Psi^S)_M (y)
\]
with $X^S\in \Gamma (TS)^H$ and $\Psi^S \in C^\infty (G\cdot S,
\fg)^G$.  Similarly
\[
X^{S'}(y') = X (y') - (\Psi^{S'})_M (y')
\]
for $y'\in S'$, $X^{S'}\in \Gamma (TS')^H$ and $\Psi^{S'} \in C^\infty
(G\cdot S, \fg)^G$.  By Lemma~\ref{lem:3.15}, for any $y\in S$,
\[
T\varphi_y (X^S(y)) = T (f(y)_M)_y (X^S(y) + \mu_M (y)),
\]
where $\mu \in C^\infty(G{\cdot}S, \fg)^G$ is defined by
\[
\mu(y) = TL_{f(y)\inv} (Tf_y (X(y)))
\]
for $y\in S$ and $\mu(g{\cdot}y) = Ad(g)\mu(y)$ for an arbitrary $g{\cdot}y\in
G{\cdot}S$.  Recall that for any $G$-invariant vector field $Z$ on
$G{\cdot}S$
\[
T(g_M) Z(y) = Z(g{\cdot }y) 
\]
for all $g\in M$ and $y\in S$.   
Therefore, for $y\in S$, $y' =f(y){\cdot} y\in S'$,
\begin{align*}
X^{S'}(y') - \varphi_*X^S(y') &= 
( X (y') -(\Psi^{S'})_M (y')) - T (f(y)_M)_y (X^S(y) + \mu_M (y))\\
&=  X (y') -(\Psi^{S'})_M (y') - T (f(y)_M)_y (X (y) - 
(\Psi^S)_M (y)) - \mu_M (y')\\
&=  X (y') -(\Psi^{S'})_M (y') - X(y') + (\Psi^S)_M (y') - \mu_M (y')\\
&= (\Psi^S - \Psi^{S'} - \mu)_M (y').
\end{align*}
Now define $\nu (y') = (\Psi^S - \Psi^{S'} - \mu) (y')$ for all $y'\in
S'$.  Since at points of the slice $S'$ the induced vector field
$\nu_M$ is tangent to $S'$, the function $\nu$ takes values in $\fh$.
\end{proof}
\noindent
This concludes our proof of Theorem~\ref{thm:1.4}.
\end{proof}

\noindent
We now turn out attention to Theorem~\ref{thm:1.5}.
\begin{proof}[Proof of Theorem~\ref{thm:1.5}]
Since $X - Y = \psi _V$, $DX(0) - DY(0) = D(\psi_V)(0)$.  Now, for any $v\in V$,
\begin{align*}
  D(\psi_V)(0)\, v
& = \left.\frac{\partial}{\partial s}\right|_{s=0} \psi_V (sv) 
= \left.\frac{\partial^2}{\partial s\partial t}\right|_{(0,0)} 
\rho (\exp t\psi_V (sv))(sv)\\ 
&=
 \left.\frac{\partial^2}{\partial s\partial t}\right|_{(0,0)} 
e^{t \delta \rho ( \psi_V (sv))}(sv) 
 = \left.\frac{\partial}{\partial s}\right|_{s=0} 
 \delta \rho (\psi (sv))\, (sv)\\ 
 &= \left(\delta \rho (\psi (sv)|_{s=0}\right) \left(
 \left.\frac{\partial}{\partial s}\right|_{s=0}(sv)\right) +
\left( \left.\frac{\partial}{\partial s}\right|_{s=0}\right)\left( \delta \rho (\psi (sv)
 \, (sv)|_{s= 0}\right) \\ 
 &= \delta \rho (\psi (0))v.
\end{align*}
\end{proof}

We finish the section with a proof of Theorem~\ref{thm:1.100}
\begin{proof}[Proof of Theorem~\ref{thm:1.100}]
  We may assume that $S_2 = T_{x_2} S_2$ and $x_2 = 0 \in T_{x_2} S_2$
  and the action of $H$ on $S_2$ is given by the slice representation
  $\rho: H\to T_{x_2} S_2$.  By Theorem~\ref{thm:1.4} there is an
  $H$-equivariant map $\psi:S_2 \to \fh_2$ so that
\[
X^{S_2} - \varphi_* X^{S_1}  = \psi_{S_2}.
\]
By Theorem~\ref{thm:1.5}
\[
D X^{S_2} (0)  - D(\varphi_* X^{S_1}) (0)  = \delta \rho (\psi (0)).
\]
Since $\psi:T_{x_2} S_2 \to \fh_2$ is $H_2$-equivariant and $0\in
T_{x_2} S_2$ is fixed by the action $H_2$, $\psi(0)$ is fixed by $H_2$
as well, i.e., $\xi:= \psi(0)\in \fh_2^{H_2}))$.  Therefore
\eqref{eq:1.100} holds.  Since $\varphi_* X^{S_1}$ is an $H_2$
invariant vector field, its differential $D(\varphi_* X^{S_1}) (0)$ is
an $H_2$-equivariant linear map from $T_{x_2} S_2$ to $T_{x_2} S_2$.
Hence $D(\varphi_* X^{S_1}) (0)$ and $\delta \rho (\xi)$ commute.
Consequently they have the same eigenvalues.  Since $D X^{S_2} (0) =
D(\varphi_* X^{S_1}) (0) + \delta \rho (\xi)$, the maps $D X^{S_2}
(x_2) = D X^{S_2} (0)$ and $D(\varphi_* X^{S_1}) (0)= T\varphi_{x_1}\,
D (X^{S_1}) (x_1)$ have the same eigenvectors as well.  Moreover, the
eigenvalues of $DX^{S_2}(0)$ are sums of the corresponding eigenvalues
of $D(\varphi_* X^{S_1}) (0)$ and $\delta \rho (\xi)$.  Since $H$ is
compact the eigenvalues of $\delta\rho (\xi)$ are purely imaginary for
all $\xi\in \fh$.  Hence the real parts of the eigenvalues of
$DX^{S_2} (0)$ and $D(\varphi_* X^{S_1}) (0)$ are the same.
\end{proof}

\section{Genericity in the space of invariant
  vector fields} \label{sec:4}

The goal of this section is to prove that the map \eqref{eq:1.8} is
part of a chain homotopy equivalence of 2-term chain complexes of
semi-topological vector spaces, Theorem~\ref{thm:???}.
Theorem~\ref{thm:1.11} follows readily from this result.
We start by recalling a fact about
Whitney $C^\infty$ topologies.
\begin{proposition}[\protect{\cite[Proposition~3.5,
    p.~46]{GG}}]\label{prop:3.5gg}
  Let $M, N, Q$ be three manifolds and $f\in C^\infty (N,Q)$. Then the
  map $f_*:C^\infty (M,N) \to C^\infty (M,Q)$ given by $f_*(\varphi)
  := f\circ \varphi$ is continuous in the Whitney $C^\infty$ topology.
\end{proposition} 
We apply it as follows.
\begin{proposition}
  Let $M$ be a manifold with an action of a Lie group $G$.  The map
\[
\partial: C^\infty (M, \fg)^G \to \Gamma(TM)^G, \quad \partial(\psi):= \psi_M
\]
is continuous in the Whitney $C^\infty$ topology (q.v.\ \eqref{eq:1'}).  In
other words $\partial: C^\infty (M, \fg)^G \to \Gamma(TM)^G$
is a  2-term chain complex of semi-topological vector spaces.
\end{proposition}
\begin{proof}
  We argue that the map $\partial$ is continuous. The action of the
  Lie group $G$ on the manifold $M$ defines a smooth map $(\cdot)_M$ from the
  trivial vector bundle $M\times \fg\to M$ to the tangent bundle
  $TM$. The map is defined by
\[
(m,X)\mapsto X_M(m):= \left.\frac{d}{dt}\right|_0 \exp (tX)\cdot m.
\]
Hence by Proposition~\ref{prop:3.5gg} the map
\[
C^\infty(M,\fg) \to \Gamma (TM), 
\qquad \psi \mapsto \psi_M = (\cdot)_M \circ \psi
\]
is continuous.  Consequently its restriction to equivariant functions
\[
(\cdot)_M: C^\infty(M,\fg)^G \to \Gamma (TM)^G
\]
is continuous as well.  
\end{proof}
\begin{theorem}\label{thm:???}
  Let $M$ be a manifold with a proper action of a Lie group $G$ and
$S\subset M$ a global slice through $x\in M$ for the action of $G$ so
that $M= G\cdot S \simeq G\times ^HS$, where $H$ is the stabilizer of $S$.
Consider the map 
\begin{equation} \label{eq:4.4}
  K_\bullet: \left(  C^\infty (S, \fh)^H \xrightarrow{\quad\partial\quad} 
    \Gamma(TS)^H\right ) \to  
\left(  C^\infty (G{\cdot }S, \fg)^G 
\xrightarrow{\quad\partial\quad} \Gamma(T (G{\cdot }S))^G\right )
\end{equation}
defined on  $C^\infty (S, \fh)^H$ by
\[
K_1(\psi)(g\cdot y) := Ad(g)\psi (y)\qquad\textrm{for all } g\in G, y\in S
\]
and on $\Gamma (TS)^H$ by 
\[
K_0(Y) (g\cdot y) = T(g_M)_y Y(y) \qquad\textrm{for all } g\in G, y\in S.
\]
 Then $K_\bullet$, which is a map of 2-term chain
complexes of semi-topological spaces, is (part of) a chain homotopy
equivalence.
\end{theorem}

\begin{proof}
Since 
\[
(K(\psi))_M (g\cdot y) = T(g_M)_y \psi_S (y),
\]
$K$ is a map of chain complexes.  We have two isomorphisms of
semi-topological vector spaces
\[
C^\infty (G\cdot S, \fg)^G \to C^\infty (S, \fg)^H \quad\textrm{and} \quad
\Gamma(T(G\cdot S))^G \to \Gamma (T(G\cdot S)|_S)^H
\]
which are both given by restrictions.  A choice of an $H$-equivariant
splitting $\fg= \fh\oplus \fm$ of the Lie algebra $\fg$ defines two
more isomorphisms:
\[
C^\infty (S, \fg)^H \simeq C^\infty (S, \fh)^H\oplus C^\infty (S,\fm)^H
\quad\textrm{and} \quad 
\Gamma (T(G\cdot S)|_S)^H \simeq
\Gamma(TS)^H \oplus C^\infty (S, \fm)^H.
\]
Next observe that the diagram 
\[
\xy
(-70,10)*+{C^\infty (S, \fh)^H}="1";
(-30,10)*+{C^\infty (S, \fh)^H \oplus C^\infty (S,\fm)^H}="2";
(20,10)*+{C^\infty (S, \fg)^H }="3";
(60,10)*+{C^\infty (G\cdot S, \fg)^G }="4";
(-70,-10)*+{\Gamma(TS)^H }="5";
(-30,-10)*+{ \Gamma(TS)^H \oplus C^\infty (S, \fm)^H}="6";
(20,-10)*+{ \Gamma (T(G\cdot S)|_S)^H }="7";
(60,-10)*+{\Gamma (T(G\cdot S))^G  }="8";
 {\ar@{->}_{\partial} "1";"5"};
 {\ar@{->}^{\partial\oplus id_{C^\infty (S,\fm)^H}} "2";"6"};
 {\ar@{->}_{\partial} "4";"8"};
 {\ar@{->}^<<<{\imath_1} "1";"2"};
{\ar@{->}^{\simeq} "2";"3"};
{\ar@{->}^{\simeq} "3";"4"};
{\ar@{->}^<<<{\imath_0} "5";"6"};
{\ar@{->}^{\simeq} "6";"7"};{\ar@{->}^{\simeq} "7";"8"};
\endxy
\]
commutes, and that the composites of the horizontal arrows are $K_1$
and $K_0$, respectively.  The projections
\[
p_1:C^\infty (S, \fh)^H \oplus C^\infty (S,\fm)^H \to  C^\infty (S, \fh)^H\quad
\textrm{and} \quad 
 p_0: \Gamma(TS)^H \oplus C^\infty (S, \fm)^H \to  \Gamma(TS)^H
\]
are continuous and  define a chain homotopy inverse to the
chain map $\imath_\bullet$.  The result follows.
\end{proof}

\section{Relation with the work of Hepworth and
  Baez-Crans}\label{sec:groupoid}

A 2-term chain complex $V_1 \xrightarrow{\partial} V_0$ of
semi-topological vector spaces defines an action of the abelian
topological group $V_1$ on $V_0$ by
\[
v_1 * v_0:= \partial (v_1) + v_0.
\]
The action is continuous since $+$ and $\partial$ are continuous.
Hence the corresponding action groupoid
\[
 V_1\times V_0 \toto V_0
\]
is a semi-topological linear groupoid, that is, a groupoid internal to
the category of semi-topological vector spaces.  The collection of all
2-term chain complexes of semi-topological vector spaces form a
2-category $\stterm$: the 1-morphisms are (continuous) chain maps and
2-morphisms are (continuous) chain homotopies.  Semi-topological
linear groupoids also form a 2-category; call it $\stGpds$: the
1-morphisms are (linear continuous) functors and 2-morphisms are
natural transformations.  The map that assigns to a 2-term chain
complex the corresponding action groupoid extends to a 2-functor
\[
{\mathcal S}_{st}:\stterm\to \stGpds.
\]
\begin{definition}\label{def:XX}
  Let $M$ be a manifold with an action of a Lie group $G$.  The {\em
    groupoid of invariant vector fields} $\XX(G\times M\toto M)$ is
  the semi-topological linear groupoid corresponding to the 2-term
  chain complex:
\[
\partial: C^\infty (M, \fg)^G \to \Gamma(TM)^G, \quad \partial(\psi):= \psi_M
\]
of semi-topological vector spaces (q.v. \eqref{eq:chain-cx}).
\end{definition}
If we forget about topology and continuity, we get the 2-categories
$\twterm$ of 2-term chain complexes of vector spaces and $\twVect$ of
groupoids internal to the category of vector spaces.  The 2-category
$\twVect$ is more commonly known as the 2-category of Baez-Crans
2-vector spaces \cite{BC}.  Just as in the case of semi-topological
vector spaces there is a  2-functor
\[
{\mathcal S}:\twterm\to \twVect,
\]
which assigns to a 2-term chain complex the corresponding action
groupoid etc.  Baez and Crans proved ({\em op.\ cit.}) that
$\mathcal{S}$ is an equivalence of 2-categories.  We note that results
of this sort had been known before \cite{Groth}, but no detailed
accounts were written down.  We also note that the analogous statement
for $\mathcal{S}_{st}:\stterm\to \stGpds$ is unlikely to be true.  The
problem is that short exact sequences of semi-topological vector
spaces need not split. One expects that the problem is fixable at a
price --- one should  enlarge the 2-category $\stterm$ to a
bicategory whose 1-morphisms are analogues of Noohi's butterflies
\cite{Noohi}.

If once again we forget topology amd assume that the group $G$ is {\sf
  compact}, then the groupoid $\XX(G\times M\toto M)$ can be
rigorously interpreted as the groupoid of vector fields on the
quotient stack $[M/G]$ (\cite[Proposition~6.1]{Hep}).  Here are the
details.

Traditionally given a proper action of a Lie group $G$ on a manifold
$M$ one thinks of a quotient $M/G$ as a topological space with some
additional structure.  For instance one proves that $M/G$ is a
stratified space and that it is ``smooth'' in an appropriate sense.
There is also another notion of a quotient that has its origins in the
works of Grothendieck, Deligne, Mumford, Artin and their collaborators
--- this is the notion of a stack quotient $[M/G]$.  Stack quotients
are instances of geometric stacks and consequently come with atlases.
A choice of an atlas for a stack defines a Lie groupoid that
``represents'' the stack.  Two different choices of atlases give rise
to Morita equivalent Lie groupoids.  It turns out that the notion of a
vector field on a stack does make sense, but instead of forming a
vector space the collection of vector fields on a given stack
${S}$ forms a category $\mathcal{X}({S})$.  It was
shown by Hepworth \cite{Hep} that if a stack ${S}$ is represented by
a Lie groupoid $\scG$ then the category of vector fields on $S$ is
equivalent to the category $\scX (\scG)$ of the so called
multiplicative vector fields on $\scG$.

We recall how the category of multiplicative vector fields, which were
defined by Mackenzie and Xu M \cite{MX}, comes about.  Given a
groupoid $\scG$ there is a tangent groupoid $T\scG$ with a canonical
functor $\pi:T\scG\to \scG$.  A {\sf multiplicative vector field} $X$
on a groupoid $\scG$ is a section of $\pi:T\scG \to \scG$.  In
particular it is a functor.  Since multiplicative vector fields are
functors, it make sense to talk about natural transformations between
them. Since vector fields take values in groupoids the natural
transformations are natural isomorphisms.  Thus multiplicative vector
fields on a Lie groupoid $\scG$ form a category (in fact, a groupoid).

Given an action of a Lie group $G$ on a manifold $M$, we have an
action Lie groupoid $G\times M\toto M$.  An invariant vector field
$X:M\to TM$ extends canonically to a multiplicative vector field
\[
(0, X): G\times M\to TG\times TM.  
\]
Consequently the groupoid $\XX (G\times M\toto M)$ is contained in the
groupoid $\scX (G\times M\toto M)$ of multiplicative vector fields.
Hepworth shows that if $G$ is compact, then inclusion is a fully
faithful and essentially surjective functor (if $G$ is not compact
this need not be the case).  Hence the groupoids $\XX (G\times M\toto
M)$ and $ \scX (G\times M\toto M)$ are equivalent.  As was mentioned
above the groupoid of multiplicative vector fields $\scX (\scG)$ on a
groupoid $\scG$, in turn, is equivalent to the groupoid of vector
fields $\mathcal{X}(B\scG)$ on the stack $B\scG$ of principal $\scG$
bundles.  Thus if two Lie groupoids $\scG$ and $\scH$ are Morita
equivalent, then the stacks $B\scG$ and $B\scH$ are isomorphic and
consequently the groupoids of multiplicative vector fields
$\scX(\scG)$ and $\scX (\scH)$ are equivalent as well.

On a manifold a vector field integrates to flow, which is, more or
less, a one-parameter family of diffeomorphisms.  Analogously on a Lie
groupoid a multiplicative vector field integrates to a one-parameter
family of functors \cite{MX} (this is imprecise; a precise statement
involves double groupoids).  An isomorphism between two multiplicative
vector fields should integrate to a one-parameter family of natural
isomorphisms between these functors.  This is what we proved in
Section~\ref{sec:2} in the special case of invariant vector fields.

If $S$ is a slice for the action of a group $G$ on a manifold $M$ and
$G\cdot S \simeq G\times ^H S$ is the corresponding tube, then the action
groupoids $G\times G\cdot S \toto G\cdot S$ and $H\times S\toto S$ are Morita
equivalent.  Consequently the groupoids $\XX (G\times G\cdot S \toto G\cdot S)$
and $\XX (H\times S \toto S)$ of invariant vector fields are Morita
equivalent (at least when $G$ is compact).  This follows from the
chain of equivalences of groupoids
\begin{align*}
\XX (H\times S\toto S) \simeq \scX(H\times S\toto S) &\simeq
\mathcal{X}([S/H]) \simeq \mathcal{X}([G \cdot S/G]) \\ 
\simeq \scX(G \times G\cdot S\toto G\cdot S) &\simeq \XX (G
\times G\cdot S\toto G\cdot S).
\end{align*}
So perhaps it is not so surprising that the natural functor
\[
K: \XX (H\times S \toto S) \to \XX (G\times
  GS \toto GS)
\]
which corresponds to the map of chain complexes $K_\bullet$ (q.v.\
\eqref{eq:4.4}) is an equivalence of categories.  I think it {\em is}
surprising that $K$ is a strong equivalence of semi-topological linear
groupoids, that is, the chain map $K_\bullet$ has a homotopy inverse.

\end{document}